\documentclass[11pt,reqno]{amsart}
\calclayout

\usepackage{breqn}

\usepackage{amssymb,latexsym}
\usepackage{mathrsfs}
\usepackage{graphicx}
\usepackage{tabularx}
\usepackage{url}		%does nice formatting of URLs

\usepackage[noabbrev,capitalise]{cleveref}

\usepackage{biblatex}
\makeatletter
\newcommand{\monthyear}[1]{%
  \def\@monthyear{\uppercase{#1}}}
\newcommand{\volnumber}[1]{%
  \def\@volnumber{\uppercase{#1}}}
\AtBeginDocument{%
\def\ps@plain{\ps@empty
  \def\@oddfoot{\@monthyear \hfil \thepage}%
  \def\@evenfoot{\thepage \hfil \@volnumber}}
\def\ps@firstpage{\ps@plain}
\def\ps@headings{\ps@empty
  \def\@evenhead{%
    \setTrue{runhead}%
    \def\thanks{\protect\thanks@warning}%
    \uppercase{The Fibonacci Quarterly}\hfil}%
  \def\@oddhead{%
    \setTrue{runhead}%
    \def\thanks{\protect\thanks@warning}%
    \hfill\uppercase{On negaFibonacci-esque Sequences And Their Relation To The Golden Ratio}}%
  \let\@mkboth\markboth
  \def\@evenfoot{%
    \thepage \hfil \@volnumber}%
  \def\@oddfoot{%
    \@monthyear \hfil \thepage}%
  }%
\footskip=25pt
\pagestyle{headings}%
}

\makeatother

\newcommand{\C}{{\mathbb C}}
\newcommand{\N}{{\mathbb N}}

\newcommand{\iu}{{i\mkern1mu}}

\theoremstyle{plain}
\numberwithin{equation}{section}
\newtheorem{thm}{Theorem}[section]
\newtheorem{theorem}[thm]{Theorem}
\newtheorem{lemma}[thm]{Lemma}
\newtheorem{corollary}[thm]{Corollary}

\newtheorem{definition}[thm]{Definition}

\newtheorem{note}[thm]{Note}

\begin{document}
%% replace the values in the next three lines by the correct information
\monthyear{August 2023}
\volnumber{Volume 61, Number 4}
\setcounter{page}{1}

\title{On negaFibonacci-esque Sequences And Their Relation To The Golden Ratio}
\author{Clemens Schütz}
\address{       Faculty of Mathematics\\
                University of Vienna\\
                Vienna, Vienna\\
                Oskar-Morgenstern-Platz 1, Austria}
\email{schuetzc48@univie.ac.at}
\thanks{We would like to thank Prof. Pitzl-Reinbacher for his support during the early stages of this paper.}
\author{Kristian Kelly}
\address{Faculty of Mathematics\\
        University of Vienna\\
        Vienna, Vienna\\
        Oskar-Morgenstern-Platz 1, Austria}
\email{kristiank22@univie.ac.at}

\begin{abstract}
The Fibonacci sequence, $F_n = F_{n - 1} + F_{n - 2}$, and its counterpart for $n < 0$, the negaFibonacci sequence, $F_{-n} = (-1)^{n + 1} \cdot F_n$, are among the most studied sequences in mathematics.
In this paper we will present a new kind of sequence, the negaFibonacci-esque sequences, identified by the property that $\omega_n = \omega_{n - 2} - \omega_{n - 1}$ with $\omega_1, \omega_2 \in \C$ chosen at will.
We will partition this kind into natural complete and complex complete negaFibonacci-esque sequences.
We will prove that $\frac{1}{\phi^n}$, referred to as the principal negaFibonacci-esque sequence, is not only a complex complete negaFibonacci-esque sequence but also one of the most significant.
Furthermore, we will present an explicit formula for all complex complete negaFibonacci-esque sequences constructed with a combination of two negaFibonacci terms.
We shall then conclude by connecting these sequences to the golden spiral and ratio.
\end{abstract}

\maketitle

\section{Introduction}
The Fibonacci sequence might be one of the most studied sequences in history, and not without reason. From the spirals of seashells to the arrangement of leaves on a stem, the Fibonacci sequence seems to emerge organically in various biological and botanical structures.
Its limit, the golden ratio \parencite{goldenfib}, finds applications in composition \parencite{fibmusic}, Pentagonal symmetry systems \parencite{pentagon}, Kepler's triangle \parencite{keplers} and  was utilized by, among others, the famed architect Le Corbusier in his work. \parencite{architecture}.
It is typically regarded as a sequence, $F_n$, only defined for positive integers. 
Its extension, however, proves to also be interesting. It is called the negaFibonacci sequence and was popularized by Donald Knuth in his famed book \textit{The Art of Computer Programming} and defined as $(-1)^{n + 1} \cdot F_n = F_{-n}$ \parencite{hyperbolic}.
He also showed its importance in his negaFibonacci Coding, an unary numeral system where each positive integer is represented by a sequence of 1s with the length of this sequence being determined by a modified negaFibonacci sequence \parencite{knuth}. 
In this paper we will present a new kind of sequence, negaFibonacci-esque, or nFe, sequences (-esque is an infix standing for 'like'), which look akin to Fibonacci sequences and are defined by $\omega_n = \omega_{n - 2} - \omega_{n - 1}$ with $\omega_1, \omega_2$ being chosen at will.
The connection to the Fibonacci numbers becomes even more apparent once we present a general explicit formula for all nFe-sequences which is constructed using two negaFibonacci terms. 
Furthermore, we will partition all nFe-sequences into so-called natural and complex complete nFe-sequences. We will then prove that the function $\frac{1}{\phi^n}$ is a complex complete nFe-sequence.
We will then use these learnings to find out more about the golden spiral \& ratio. 
\subsection{Definitions}\leavevmode\newline
We shall begin this paper with a couple of definitions that will be required during the proofs of this paper.\\

\begin{definition}
        The term nFe-sequence shall be an abbreviation for negaFibonacci-esque sequence.
\end{definition}

We shall define the golden ratio, $\phi$, for completeness sake.
\begin{definition}
        $\phi = \frac{1 + \sqrt{5}}{2} \approx 1.618033988749894\dots$
\end{definition}
Should one ever have time on their hands, it proves to be an interesting task to try and find identities involving this number such as $\phi^2 = \phi + 1, \frac{1}{\phi} = \phi - 1$.
Some of these identities will be used during this paper. We will not be proving them and, should one be interested, recommend them as exercises for the reader.

\begin{definition}
        The Fibonacci sequence is a sequence given by $F_1 = 1$, $F_2 = 1$, $F_n = F_{n - 2} + F_{n - 1}$.
\end{definition}

\begin{definition}
        The negaFibonacci numbers are a sequence given by $F_{-n} = (-1)^{n + 1} \cdot F_n$
\end{definition}

\begin{definition}
        The Fibonacci sequence can be given using an explicit formula:  \[\Theta: \C \to \C\] \[\Theta(n) = \frac{\phi^n - (-\phi)^{-n}}{\sqrt{5}} = F_n, \forall n \in \C\]
\end{definition}
This formula is based on Binet's Formula \parencite{binetform} and will not be proven within the scope of this paper. However, the proof may be found in \parencite{binetproof}.
Using this continuous extension of the Fibonacci numbers we are able to get Fibonacci numbers for even imaginary n.

\begin{definition}\label{extension}
        \[I(n): \C \to \C\] 
        \[I(n) = (-1)^n\]
        For the extension of the domain of $(-1)^n$ to the complex numbers, we will use the following definition:
        \[(-1)^n = \exp(\iu \cdot \pi \cdot n)\]
        Which is well-defined. 
\end{definition}
This function shall merely be referred to as the identity function. 
This name is generally ambiguous. Nevertheless, in the scope of this paper only this definition is used.

\begin{note}
        \cref{extension} is also used, when rasing any negative number to a complex power. For example:
        \[(-\phi)^n = (-1)^n \cdot \phi^n = \exp(\iu \cdot \pi \cdot n) \cdot \phi^n\]
\end{note}

\section{negaFibonacci-esque Sequences}
\begin{definition}\label{sequence}
        A sequence $\omega_n$ shall be called a natural complete negaFibonacci-esque, or nFe, if, for given $\omega_1$ and $\omega_2$:
        \[\omega_n = \omega_{n - 2} - \omega_{n - 1}\]
        \[\forall n \in \N\]
\end{definition}
As $\omega_n$ is only defined for $n \in \N$, our property, $\omega_n = \omega_{n - 2} - \omega_{n - 1}$, must only be fulfilled for $n \in \N$.
For this reason, we shall now define a kind of function, for which this property holds for all $n \in \C$:
\begin{definition}\label{comseq}
        Let:
        \[\tilde{\omega}: \C \to \C\]
        $\tilde{\omega}(n)$ shall be called a complex complete nFe-sequence if:
        \[\forall n \in \C: \tilde{\omega}(n) = \tilde{\omega}(n - 2) - \tilde{\omega}(n - 1)\]
        The natural complete nFe-sequence, $\omega_n$, such that:
        \[\forall n \in \N: \tilde{\omega}(n) = \omega_n\]
        shall be called the corresponding natural complete nFe-sequence.
\end{definition}

\begin{lemma}
        There is an uncountable infinite amount of natural complete nFe-sequences and complex complete sequences.
\end{lemma}

\begin{proof}
        Both proofs are trivial in ZFC (as $card(A \times A) = card(A)$, where $A$ is uncountably infinite, implies the axiom of choice) and only here for the sake of completeness.\\
        Let $W$ be the set of all natural complete nFe-sequences. Naturally, a sequence is determined only by its initial values, $\omega_1, \omega_2$. As both of these values are real:
        \begin{equation}
                card(W) = card(\C \times \C) = card(\C)
        \end{equation}
        Let $\tilde{W}$ be the set of all complex complete nFe-sequences. We shall start with a lower bound: As per \cref{implication}, there is a complex complete nFe-sequence for every pair of indices $(\eta,\gamma)$:
        \begin{equation}
                card(\tilde{W}) \geq card(\C \times \C) = card(\C)
        \end{equation}
        We can quickly get an upper bound as well. Let $\mathscr{F}_{\C,\C}$ denote the set of all functions, $f: \C \to \C$. Naturally:
        \begin{equation}
                card(\tilde{W}) \leq card(\mathscr{F}_{\C, \C}) = card(\C)
        \end{equation}
        Thus:
        \begin{equation}
                card(\C) \leq card(\tilde{W}) \leq card(\C)
        \end{equation}
        Thus, we have successfully proven that $card(\tilde{W}) = card(\C)$.
\end{proof}

It might be worth noting that this implies that the cardinality of all natural complete nFe-sequences is equal to the cardinality of all complex complete nFe-sequences. 

\section{An Explicit Formulation}

For the purpose of creating an explicit formula for all such sequences, we shall have two goals in mind.
First of all, the explicit formula has to be a complex complete nFe-sequence for all possible combinations of initial values, $\Omega_{\eta,\gamma}(1)$ and $\Omega_{\eta,\gamma}(2)$.
Secondly, we want to be able to, given a natural complete nFe-sequence, $\omega_n$, find a complex complete sequence, $\Omega_{\eta,\gamma}$, for which $\omega_n$ is its corresponding natural complete sequence.
We shall accomplish our goals by identifying each complex complete nFe-sequence with two indices, $\eta$ and $\gamma$.
\begin{definition}
        Let:         \[\Omega_{\eta,\gamma}: \C \to \C\]
        \[\Omega_{\eta,\gamma}(n) = \eta \cdot I(n - 1) \cdot F_n + \gamma \cdot I(n) \cdot F_{n - 1} \cdot \phi\] where: \[I(n) = (-1)^n\] \[\forall n \in \C\]
        For purposes of continuity:
        \[\Omega_{\eta,\gamma}(n) = \eta \cdot I(n - 1) \cdot \Theta(n) + \gamma \cdot I(n) \cdot \Theta(n - 1) \cdot \phi\]
        Where:
        \[\Theta(n) = \frac{\phi^n - (-\phi)^{-n}}{\sqrt{5}}\]
\end{definition}

\begin{note}
        Note that the omega function is equivalent to: \[\Omega_{\eta,\gamma}(n) = \eta \cdot F_{-n} + \gamma \cdot F_{-n + 1} \cdot \phi\]
\end{note}

As mentioned before, we are indexing $\Omega$ at two values, $\eta, \gamma$. What is still left to do is to define how, given a natural complete nFe-sequence, $\omega_n$, these values have to be chosen such that $\Omega_{\eta,\gamma}(1) = \omega_1$ and $\Omega_{\eta,\gamma}(2) = \omega_2$.
\begin{note}\label{coeffs}
        $\Omega_{\eta, \gamma}(n)$ shall be indexed at two positions, the so-called nFe-coefficients.\\
        If we are given two initial values, $\omega_1, \omega_2$, where we want \[\Omega_{\eta,\gamma}(1) = \omega_1\] and \[\Omega_{\eta,\gamma}(2) = \omega_2\] \\
        nFe-$\eta$ shall be given by \[\eta = \omega_1\]
        nFe-$\gamma$ shall be given by \[\gamma = \frac{\omega_2 + \omega_1}{\phi}\]
\end{note}

\begin{theorem}\label{comcomplete}
        $\Omega_{\eta, \gamma}(n)$ is a complex complete nFe-sequence $\forall \eta, \gamma \in \C$.
\end{theorem}

\begin{proof}
        It stands to prove that, $\forall \eta, \gamma \in \C$:
        \begin{equation}
                \Omega_{\eta,\gamma}(n) = \Omega_{\eta,\gamma}(n - 2) - \Omega_{\eta,\gamma}(n - 1)
        \end{equation}
        We start by expanding the equation:
        \begin{dmath}
                \eta \cdot (-1)^{n - 1} \cdot \frac{\phi^n - (-\phi)^{-n}}{\sqrt{5}} + \gamma \cdot (-1)^n \cdot \frac{\phi^{n - 1} - (-\phi)^{-n + 1}}{\sqrt{5}} \cdot \phi = \eta \cdot (-1)^{n - 1} \cdot \frac{\phi^{n - 2} - (-\phi)^{-n + 2}}{\sqrt{5}} + \gamma \cdot (-1)^n \cdot \frac{\phi^{n - 3} - (-\phi)^{-n + 3}}{\sqrt{5}} \cdot \phi - \left[ \eta \cdot (-1)^{n} \cdot \frac{\phi^{n - 1} - (-\phi)^{-n + 1}}{\sqrt{5}} + \gamma \cdot (-1)^{n - 1} \cdot \frac{\phi^{n - 2} - (-\phi)^{-n + 2}}{\sqrt{5}} \cdot \phi \right]
        \end{dmath}
        We simplify, multiply with $\sqrt{5}$, divide by $(-1)^n$ and simplify once again:
        \begin{dmath}
        \eta \cdot (-\phi^n + (-\phi)^{-n}) + \gamma \cdot (\phi^n - (-\phi)^{-n + 1} \cdot \phi) = \eta \cdot (-\phi^{n - 2} + (-\phi)^{-n + 2}) + \gamma \cdot (\phi^{n - 2} - (-\phi)^{-n + 3} \cdot \phi) + \eta \cdot (-\phi^{n - 1} + (-\phi)^{-n + 1}) + \gamma \cdot (\phi^{n - 1} - (-\phi)^{-n + 2} \cdot \phi)
        \end{dmath}
        We continue by splitting up the negative bases, $(-\phi)^n = (-1)^n \cdot \phi^n$, continue to simplify and combine parenthesis with the same factor:
        \begin{dmath}
                    \eta \cdot (-\phi^n + (-1)^{-n} \cdot \phi^{-n}) + \gamma \cdot (\phi^n - (-1)^{-n + 1} \cdot \phi^{-n + 2}) = \eta \cdot (-\phi^{n - 2} + (-1)^{-n + 2} \cdot \phi^{-n + 2} -\phi^{n - 1} + (-1)^{-n + 1} \cdot \phi^{-n + 1}) + \gamma \cdot (\phi^{n - 2} - (-1)^{-n + 3} \cdot \phi^{-n + 4} + \phi^{n - 1} - (-1)^{-n + 2} \cdot \phi^{-n + 3})
        \end{dmath}
        We now multiply with $\phi^n$:
        \begin{dmath}
                \eta \cdot (-\phi^{2n} + (-1)^{n}) + \gamma \cdot (\phi^{2n} - (-1)^{n + 1} \cdot \phi^2) = \eta \cdot (-\phi^{2n - 2} + (-1)^{n} \cdot \phi^2 - \phi^{2n - 1} + (-1)^{n + 1} \cdot \phi) + \gamma \cdot (\phi^{2n - 2} - (-1)^{n + 1} \cdot \phi^{4} + \phi^{2n - 1} - (-1)^n \cdot \phi^{3})
        \end{dmath}
        We simplify the left hand side of the equation:
        \begin{dmath}
                \gamma \cdot \phi^{2n} - \eta \cdot \phi^{2n} + ((\phi + 1) \cdot \gamma + \eta) \cdot (-1)^n = \eta \cdot (-\phi^{2n - 2} + (-1)^{n} \cdot \phi^2 - \phi^{2n - 1} + (-1)^{n + 1} \cdot \phi) + \gamma \cdot (\phi^{2n - 2} - (-1)^{n + 1} \cdot \phi^{4} + \phi^{2n - 1} - (-1)^n \cdot \phi^{3})
        \end{dmath}
        We continue to apply the same process of simplification to both parenthesis on the right hand side:
        \begin{dmath}
                \gamma \cdot \phi^{2n} - \eta \cdot \phi^{2n} + ((\phi + 1) \cdot \gamma + \eta) \cdot (-1)^n = \eta \cdot (-1)^n - \eta \cdot \phi^{2n} + \gamma \cdot \phi^{2n} - \gamma \cdot \phi^2 \cdot (-1)^n
        \end{dmath}
        Cancelling out the terms leaves us with:
        \begin{equation}
                0 = 0
        \end{equation}
\end{proof}

\begin{theorem}\label{implication}
        There at least one complex complete nFe-sequence for every natural complete nFe-sequence $\omega_n$.  \\Furthermore, there exists exactly one corresponding natural complete nFe-Sequence, as defined in \cref{comseq}, for every complex complete nFe-sequence.
\end{theorem}

\begin{proof}
        We start by proving the first part of the theorem, that there exists at least one complex complete nFe-sequence for every natural complete nFe-sequence.
        Let $\omega_1, \omega_2$ denote the initial values of $\omega_n$. We now calculate our indices according to \cref{coeffs}:
        \begin{equation}
                \tilde{\eta} = \omega_1
        \end{equation}
        and
        \begin{equation}
                \tilde{\gamma} = \frac{\omega_1 + \omega_2}{\phi}
        \end{equation}
        Let us prove that $\Omega_{\tilde{\eta}, \tilde{\gamma}}(1) = \omega_1$ and $\Omega_{\tilde{\eta}, \tilde{\gamma}}(2) = \omega_2$\\
        For $n = 1$ we quickly arrive at:
        \begin{dmath}
                \tilde{\eta} \cdot (-1)^0 \cdot F_1 +  \tilde{\gamma} \cdot (-1)^1 \cdot F_0 \cdot \phi = \tilde{\eta} + 0 \cdot  \tilde{\gamma} = \tilde{\eta} = \omega_1
        \end{dmath}
        $n = 2$:
        \begin{dmath}
                \tilde{\eta} \cdot (-1)^1 \cdot F_2 +  \tilde{\gamma} \cdot (-1)^2 \cdot F_1 \cdot \phi = -\tilde{\eta} +  \tilde{\gamma} \cdot \phi = -\tilde{\eta} + \frac{\omega_1 + \omega_2}{\phi} \cdot \phi = -\tilde{\eta} + \omega_1 + \omega_2 = -\omega_1 + \omega_1 + \omega_2 = \omega_2
        \end{dmath}
        As $\tilde{\eta},  \tilde{\gamma} \in \C$, according to \cref{comcomplete}, $\Omega_{\tilde{\eta}, \tilde{\gamma}}(n)$ is complex complete, thus:
        \begin{equation}
                \Omega_{\tilde{\eta}, \tilde{\gamma}}(n) = \Omega_{\tilde{\eta}, \tilde{\gamma}}(n - 2) - \Omega_{\tilde{\eta}, \tilde{\gamma}}(n - 1)
        \end{equation}
        for all values $n \in \C$ and thus also for all $n \in \N$. Thus:
        \begin{equation}
                \Omega_{\tilde{\eta}, \tilde{\gamma}}(n) = \omega_n, \forall n \in \N
        \end{equation}
        We continue with the second part of the statement. Given any $\Omega_{\eta,\gamma}(n)$, if we restrict $n$ to the natural numbers, we get a sequence, $\omega_n$, which, naturally, has to obey $\omega_n = \omega_{n - 2} - \omega_{n - 1}$. Thus, it is a natural complete nFe-sequence. This is trivially the only possible corresponding natural complete nFe-sequence of $\Omega_{\eta,\gamma}$.\\
\end{proof}

\section{The Principal negaFibonacci-esque sequence}

\begin{definition}\label{principal}
        $\frac{1}{\phi^n}$ shall be called the principal nFe-sequence.\\
        However, for indexing's sake (any fraction is 0-indexed yet sequences are 1-indexed), we shall let:
        \[\Phi(n) = \frac{1}{\phi^{n - 1}}\] be the delegate of the sequence, $\Phi_n$, with initial values:
        \[\omega_1 = 1\]
        \[\omega_2 = \phi - 1\]
        and is identified by $\eta = \gamma = 1$:
        \[\Omega_{1,1}(n)\]
\end{definition}

This principal nFe-sequence allows us to write all powers of $\phi$ in terms of negaFibonacci numbers, akin to Binet's formula, as proven in \cref{binnet}.

\newcolumntype{Y}{>{\centering\arraybackslash}X}
\setlength{\extrarowheight}{3pt}
\begin{center}
        \centering
        \begin{tabularx}{\textwidth}{|Y||Y|Y|}
                        \hline
                        n-Value & Power of $\phi$ & negaFibonacci-form \\
                        \hline
                        $n = -8$ & $\phi^9$ & $34\phi + 21$ \\
                        \hline
                        $n = -7$ & $\phi^8$ & $21\phi + 13$ \\
                        \hline
                        $n = -6$ &$\phi^7$ & $13\phi + 8$ \\
                        \hline
                        $n = -5$ &$\phi^6$ & $8\phi + 5$ \\
                        \hline
                        $n = -4$&  $\phi^5$ & $5\phi + 3$ \\
                        \hline
                        $n = -3$ & $\phi^4$ & $3\phi + 2$ \\
                        \hline
                        $n = -2$ & $\phi^3$ & $2\phi + 1$ \\
                        \hline
                        $n = -1$ & $\phi^2$ & $\phi + 1$ \\
                        \hline
                        $n = 0$ & $\phi^1$ & $\phi$ \\
                        \hline
                        $n = 1$ & $\phi^{0}$ & $1$ \\
                        \hline
                        $n = 2$ & $\phi^{-1}$ & $-1 + \phi$ \\
                        \hline
                        $n = 3$ & $\phi^{-2}$ & $2 - \phi$ \\
                        \hline
                        $n = 4$ &$\phi^{-3}$ & $-3 + 2\phi$ \\
                        \hline
                        $n = 5$ &$\phi^{-4}$ & $5 - 3\phi$ \\
                        \hline
                        $n = 6$ &$\phi^{-5}$ & $-8 + 5\phi$ \\
                        \hline
                        $n = 7$ &$\phi^{-6}$ & $13 - 8\phi$ \\
                        \hline
                        $n = 8$ &$\phi^{-7}$ & $-21 + 13\phi$ \\
                \hline
        \end{tabularx}
        \centering
\end{center}

To avoid repetition, we shall start by proving that $\Phi(n)$ is in fact equivalent to $\Omega_{1,1}(n), \forall n \in \C$, thus being a complex complete nFe-sequence. Then  we will use this theorem to prove that it is its own corresponding natural complete nFe-sequence.

\begin{theorem}\label{omegaeqphi}
        The principal nFe-sequence is a complex complete nFe-sequence and equivalent to the $\Omega_{1,1}(n)$ sequence. \[\Omega_{1,1}(n) = \Phi(n), \forall n \in \C\]
\end{theorem}

\begin{proof}
        As this is the principal sequence, the nFe-coefficients are $\eta = \gamma = 1$:
        \begin{equation}
                I(n - 1) \cdot \Theta(n) + I(n) \cdot \Theta(n - 1) \cdot \phi
        \end{equation}
        Upon expansion:
        \begin{equation}
                (-1)^{n - 1} \cdot \frac{\phi^n - (-\phi)^{-n}}{\sqrt{5}} + (-1)^n \cdot \frac{\phi^{n - 1} - (-\phi)^{-n + 1}}{\sqrt{5}} \cdot \phi = \frac{1}{\phi^{n - 1}}
        \end{equation}
        Extend the right side with $\sqrt{5}$, then multiply by $\sqrt{5}$ and divide by $(-1)^n$, upon which the equation simplifies to:
        \begin{equation}
                -\phi^n + (-\phi)^{-n} + \phi^n - (-\phi)^{-n + 1} \cdot \phi = \frac{\sqrt{5}}{(-1)^n \cdot \phi^{n - 1}}
        \end{equation}
        Which, simplified, is equivalent to:
        \begin{equation}
                1 - (-\phi) \cdot \phi = \frac{\sqrt{5}}{(-1)^n \cdot \phi^{n - 1} \cdot (-\phi)^{-n}}
        \end{equation}
        Now, we just trivially need to expand all of these terms:
        \begin{equation}
                1 - (-\phi) \cdot \phi = \frac{\sqrt{5}}{(-1)^n \cdot \phi^n \cdot \phi^{-1} \cdot \phi^{-n} \cdot (-1)^{-n}} = \frac{\sqrt{5} \cdot \phi}{(-1)^0} = \sqrt{5} \cdot \phi
        \end{equation}
        The left side can be simplified even further:
        \begin{equation}
                1 + \phi^2 = \sqrt{5} \cdot \phi
        \end{equation}
        This is identity is completely trivial, yet we have included the full proof here for completeness sake. To start, simply divide by $\phi$:
        \begin{equation}
                \frac{1}{\phi} + \phi = \sqrt{5}
        \end{equation}
        Recall: $\frac{1}{\phi} = \phi - 1$. Thus:
        \begin{equation}
                2\phi - 1 = \sqrt{5}
        \end{equation}
        Which is trivial, considering $\phi = \frac{1 + \sqrt{5}}{2}$.
\end{proof}

\begin{corollary}
        The principal nFe-sequence, $\Phi_n = \Phi(n)$, $n \in \N$ is the corresponding natural complete nFe-sequence of the complex complete nFe-sequence $\Phi(n)$. $\Phi_n$ shall have the initial values:
        \[\omega_1 = 1\]
        \[\omega_2 = \phi - 1\] 
\end{corollary}

\begin{proof}
        $\Phi(n)$, as defined in \cref{principal}, is complex complete according to \cref{omegaeqphi}. Any complex complete sequence has, as per \cref{implication}, a corresponding natural complete nFe-sequence whose initial values, $\tilde{\omega}_1, \tilde{\omega}_2$, are:
        \begin{equation}
                \tilde{\omega}_1 = \eta = 1
        \end{equation}
        and:
        \begin{equation}
                \tilde{\omega}_2 = \phi \cdot \gamma - \tilde{\omega_1} = \phi - 1
        \end{equation}
        As:
        \begin{equation}
                \tilde{\omega}_1 = \omega_1 \land \tilde{\omega_2} = \omega_2
        \end{equation}
        And as a natural complete nFe-sequence is only determined by its two initial values, it is proven.
\end{proof}

\section{Connections to the Golden Spiral and Ratio}
\begin{figure}[h]
        \includegraphics[width=\textwidth/2]{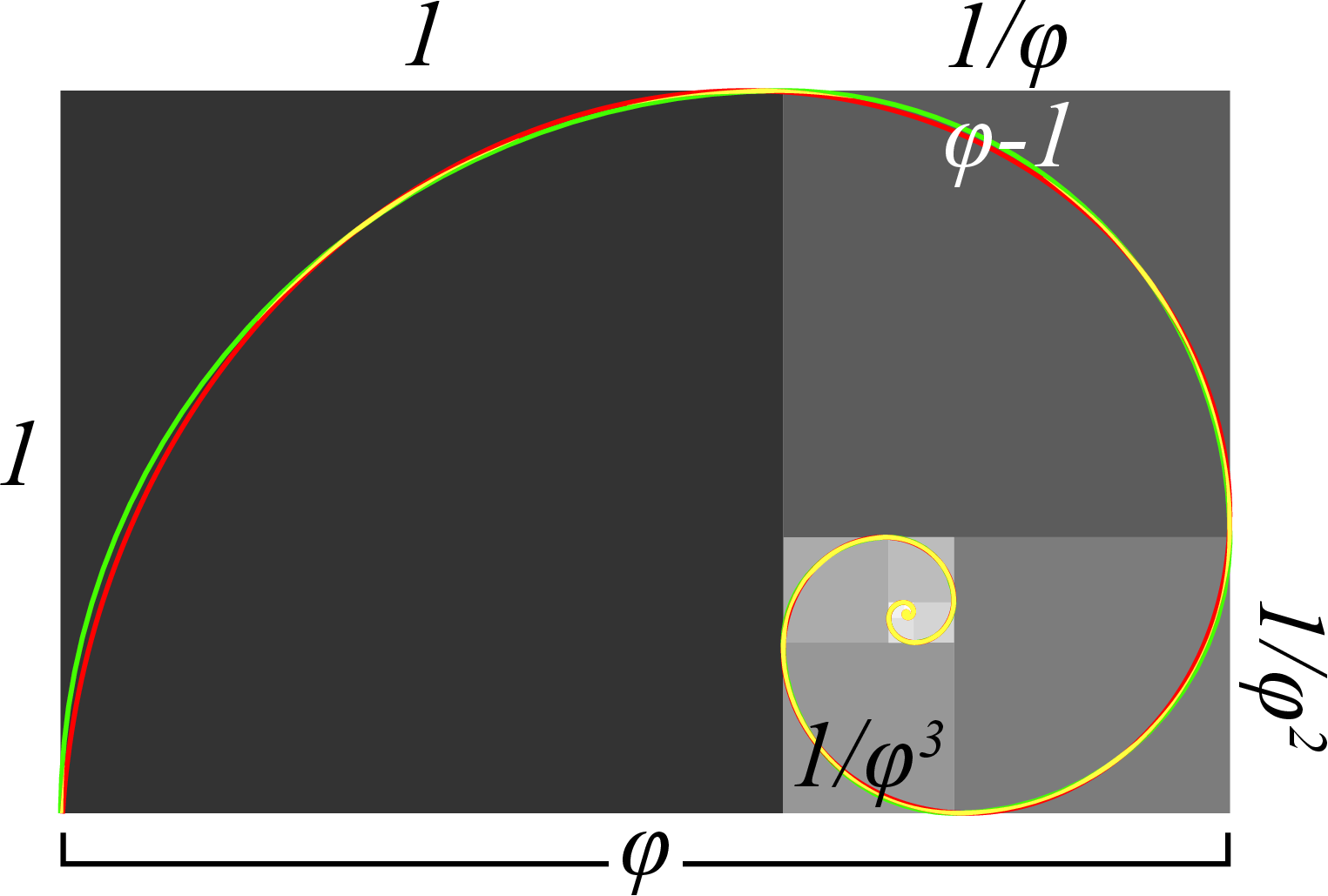}
        \centering
        \caption[Visualization of the golden spiral]{The above image visualizes the golden spiral and its growth factor \parencite{FakeRealLogSpiral}.}
        \label{spiralimg}
\end{figure}
\leavevmode\newline
When looking at the golden spiral in its segmented form, as visualized by \cref{spiralimg}, one quickly sees that the length of each segment is given by $\frac{1}{\phi^{n - 1}}$.\\
\begin{definition}
        We shall split the entirety of the golden spiral into so-called segments, where the n'th segment ($n \geq 1$) shall have a length of $\frac{1}{\phi^{n - 1}}$
\end{definition}
Naturally, this is an equivalent statement to the fact that the golden spiral has an inwards growth factor of $\phi$.
\begin{corollary}
        For $n > 2$, the length of the n'th segment will be the length of the segment before the previous one subtracted by the length of the previous one.
\end{corollary}
\begin{proof}
        This follows directly from \cref{omegaeqphi}. There exists a graphic proof which is left to the reader as an exercise.
\end{proof}

\begin{corollary}\label{phithing}
        $\frac{1}{\phi^n} = \frac{1}{\phi^{n - 2}} - \frac{1}{\phi^{n - 1}}$
\end{corollary}
\begin{proof}
        This proof follows directly from \cref{omegaeqphi}.
\end{proof}

\begin{corollary}
        $\phi^{n} = \phi^{n + 2} - \phi^{n + 1}$
\end{corollary}

\begin{proof}
        This proof follows directly from \cref{phithing}.
\end{proof}

Furthermore, we can draw a connection between the $\Omega_{1,1}(n)$ function and Binet's formula for powers of $\phi$.
\begin{lemma}\label{binnet}
        Binet's formula for the powers of the golden ratio, $\phi^n = F_n \cdot \phi + F_{n - 1}$, can be derived from the $\Omega_{1,1}(n)$ function.
\end{lemma}

\begin{proof}
        We start with the definition:
        \begin{equation}
                \Omega_{1,1}(n) = (-1)^{n - 1} F_n + (-1)^n F_{n - 1} \phi
        \end{equation}
        We apply the property that:
        \begin{equation}
                F_{-n} = (-1)^{n + 1} \cdot F_n
        \end{equation}
        And finish the derivation:
        \begin{equation}
                \Omega_{1,1}(n) = F_{-n} + F_{-n + 1} \cdot \phi = \phi^{-n + 1}
        \end{equation}
        \begin{equation}
                \phi^n = F_n \cdot \phi + F_{n - 1}
        \end{equation}
\end{proof}

\medskip
\printbibliography

\noindent MSC2020: 11B39, 33C05

\end{document}